\documentclass[11pt,a4paper]{amsart}
\usepackage{geometry}
\usepackage[english]{babel}
\usepackage{amssymb,enumerate,latexsym,hyperref,color,graphicx,cite}

\newenvironment{proof*}[1]{\medskip\noindent\textbf{#1\ }}{\hspace*{\fill}$\Box$\medskip}
\newtheorem{theorem}{Theorem}[section]
\newtheorem{lemma}[theorem]{Lemma}

\theoremstyle{remark}
\newtheorem{remark}[theorem]{Remark}
\newtheorem{example}[theorem]{Example}
\newtheorem{problem}[theorem]{Problem}
\newtheorem{conjecture}[theorem]{Conjecture}
\newtheorem{question}[theorem]{Question}

\newtheorem{mainthm}{\bf Theorem}

\newtheorem{coralph}[mainthm]{\bf Corollary}
\newtheorem{mainexample}[mainthm]{\bf Example}

\newcommand{\B}{{\mathbf B}}

\newcommand{\Z}{{\mathbf Z}}

\newcommand{\CC}{{\mathbb C}}

\newcommand{\QQ}{{\mathbb Q}}
\newcommand{\RR}{{\mathbb R}}

\newcommand{\TT}{{\mathbb T}}

\newcommand{\ZZ}{{\mathbb Z}}

\newcommand{\SL}{{\rm SL}}

\newcommand{\be}[1]{\begin{equation} \label{#1} }
\newcommand{\ee}{\end{equation}}
\newcommand{\beq}{\begin{equation}}

\def \Diff{{\rm Diff}}

\def \hx0{\hat{x_0}}

\def \B{\mathcal{B}}

\def \Z{\mathcal{Z}}

\begin{document}
\author{Jiesong Zhang}
\address{Department of Mathematics\\ Kungliga Tekniska högskolan, Lindstedtsvägen 25\\ SE-100 44 Stockholm\\ Sweden}
\email{jiesongz@kth.se}
\title[Invariant complex structures]{Invariant complex structures for affine automorphisms: a cocycle viewpoint}

\begin{abstract}
We prove that if a holomorphic diffeomorphism of a compact complex manifold is bi-Lipschitz conjugate to an ergodic affine automorphism \(A\) on \(\Gamma\backslash G\), then the conjugacy is \(C^\infty\). Moreover, if \(A\) is weakly mixing, then the induced complex structure on \(\Gamma\backslash G\) is left-invariant. As applications, we establish a regularity bootstrap result for holomorphic Anosov diffeomorphisms bi-Lipschitz conjugate to affine models, as well as a holomorphic analogue of the rigidity theorem for higher-rank abelian Anosov actions by Hertz--Wang \cite{hw14}. 

The key observation is that the condition for a diffeomorphism to preserve a complex structure has the same form as the cocycle compatibility relation appearing in the study of centralizers. This places invariant complex structures and centralizers within a common \(\mathbb Z^2\)-cocycle framework. From this viewpoint, our main result may be regarded as a holomorphic counterpart of the Lipschitz centralizer rigidity theorem of Damjanovi\'c--Wilkinson--Wu--Xu \cite{dwwx25} for affine automorphisms.

\end{abstract}

\maketitle

\section{Introduction}
An important problem in complex geometry is to determine which smooth real manifolds admit complex structures, and to understand the space of all such structures. A classical and still open example is the Hopf problem, which asks whether \(S^6\) admits a complex structure. In holomorphic dynamics, it is natural to consider a dynamical analogue: given a smooth diffeomorphism \(f\colon M\to M\), what are the complex structures on \(M\) for which \(f\) becomes holomorphic?
\begin{problem}[Invariant complex structures]\label{problem 1}
Let \(M\) be an even-dimensional smooth real manifold admitting complex structures, i.e., the space \(\mathcal{J}_M\) of complex structures on \(M\) is nonempty. Let \(f\colon M\to M\) be a \(C^\infty\) diffeomorphism. Describe the subset
\[
    \mathcal{J}_M(f)\subset \mathcal{J}_M
\]
consisting of \(f\)-invariant complex structures on \(M\), or equivalently, the space of
complex structures on \(M\) with respect to which \(f\) is holomorphic.
\end{problem}

This problem is particularly interesting when one imposes dynamical assumptions on \(f\). For instance, if \(f\colon M\to M\) is a diffeomorphism of a compact \(4\)-manifold with positive topological entropy, then any complex structure with respect to which \(f\) is holomorphic must be Kähler, and the resulting complex surfaces can be further classified up to birational equivalence \cite[Theorem~10.1]{can14}.

Another important case is given by Anosov systems. Ghys conjectured \cite{ghy95} that every holomorphic Anosov diffeomorphism is holomorphically conjugate to a holomorphic diffeomorphism of a complex infra-nilmanifold. This conjecture was proved by Ghys \cite{ghy95} for complex surfaces, and by Cantat \cite{can04} for projective varieties under the additional assumption that the stable and unstable distributions \(E^s\) and \(E^u\) are holomorphic. From the viewpoint of Problem~\ref{problem 1}, this conjecture can be formulated as follows: if an Anosov diffeomorphism \(f\colon M\to M\) admits an invariant complex structure \(J\), then \((M,J)\) should be biholomorphic to a complex infra-nilmanifold, and \(f\) should be biholomorphically conjugate to a holomorphic affine automorphism.

Problem \ref{problem 1} is challenging even in specific settings.  For example, it remains open whether a compact complex manifold diffeomorphic to \(\mathbb T^{2d}\), \(d\geq 3\), and supporting a holomorphic Anosov diffeomorphism must be a complex torus\footnote{For $d \leq 2$, this is always true by uniformization theorem and Kodaira classification.}. This difficulty is also illustrated by an example of Xu--Zhang \cite{xz25}. For every \(d\geq 5\), they construct a linear partially hyperbolic automorphism  \(L:\TT^{2d} \to \TT^{2d}\) which is holomorphic both for the standard complex  structure and for a non-standard complex structure \(J\). In the latter case, the resulting complex manifold is not Kähler, and hence is not a complex torus.

A key feature of the linear maps appearing in the example of Xu--Zhang is that they are not ergodic and hence not weakly mixing. In this note, we prove a rigidity result showing that this phenomenon cannot occur in the weakly mixing case:  if \(A\colon \Gamma\backslash G\to \Gamma\backslash G\) is weakly mixing and holomorphic with respect to a complex structure $J$, then $J$ is a left-invariant complex structure (see Section~\ref{section complex nilmanifold} for the definition). This gives an answer to Problem \ref{problem 1} for weakly mixing affine automorphisms. 

The proof is based on a simple observation: once Problem~\ref{problem 1}
is formulated in the language of \textit{cocycles over abelian actions}, it becomes closely related to the following classical centralizer problem.

\begin{problem}[Centralizer]\label{problem 2}
Let \(f\colon M\to M\) be a \(C^\infty\) diffeomorphism of a compact manifold.
Describe the \(C^\infty\)-centralizer of \(f\), namely
\[
\Z^\infty(f):=\{g\in \Diff^\infty(M): g\circ f=f\circ g\}.
\]
\end{problem}

\subsection{$\ZZ^2$-cocycles: centralizers and invariant complex structures}
\label{sec central}
We now explain the relation between centralizers and invariant complex structures from the cocycle point of view. Let \(\alpha\colon \mathbb Z^2\to \Diff^\infty(M)\) be an abelian action, and let \(G\) be a group. A \(G\)-valued cocycle over \(\alpha\) is a map
\[
\B\colon \mathbb Z^2\times M\to G
\]
satisfying the cocycle identity
\[
\B(\mathbf n+\mathbf m,x)
=
\B(\mathbf n,\alpha^{\mathbf m}x)\,\B(\mathbf m,x),
\qquad
\mathbf n,\mathbf m\in \mathbb Z^2,\ x\in M.
\]
Equivalently, if \(\alpha\) is generated by two commuting maps 
\[
f=\alpha^{(1,0)},\qquad g=\alpha^{(0,1)},
\]
then the cocycle is determined by its values over the generators,
\[
\B_f(x):=\B((1,0),x),\qquad
\B_g(x):=\B((0,1),x),
\]
satisfying
\begin{equation}\label{eq compatibility}
    \B_f(gx)\,\B_g(x)=\B_g(fx)\,\B_f(x).
\end{equation}
\begin{example}[Centralizers]\label{ex centralizer}
Let \(f\colon M\to M\) be a \(C^\infty\) diffeomorphism of a compact manifold. If \(g\in \Z^\infty(f)\), then the relation \(f\circ g=g\circ f\) implies
\[
D f_{gx}\circ Dg_x
=
Dg_{fx}\circ Df_x .
\]
Thus the pairs
\[
(f,\B_f)=(f,Df),
\qquad
(g,\B_g)=(g,Dg).
\]
satisfy equation \eqref{eq compatibility}.
\end{example}

\begin{example}[Invariant complex structures]\label{ex complex-structure}
Let \(f\colon M\to M\) be a \(C^\infty\) diffeomorphism of a compact manifold. Suppose that \(f\) is holomorphic with respect to a complex structure \(J\colon TM\to TM\). Then
\[
Df_x\circ J_x=J_{fx}\circ Df_x .
\]
In other words, the holomorphicity condition is precisely equation \eqref{eq compatibility} for the two commuting maps \(f\) and \(\mathrm{Id}_M\), with
\[
(f,\B_f)=(f,Df),
\qquad
(g,\B_g)=(\mathrm{Id}_M,J).
\]
\end{example}
An interesting case is that when $f=A:\Gamma\backslash G\to \Gamma\backslash G$ is an affine automorphism, where $G$ is a simply connected real Lie group, and \(\Gamma<G\) is a cocompact lattice. In \cite{dwwx25}, the authors study the Lipschitz centralizer
\[\Z^{\mathrm{Lip}}(A):=\{f\in \mathrm{Homeo}(\Gamma\backslash G): f \text{ is bi-Lipschitz and } f\circ A=A\circ f\}\]
of $A$ and show that
\begin{itemize}
    \item if $A$ is ergodic with respect to the Haar measure $m$, then $\Z^{\mathrm{Lip}}(A) \subset \Diff^\infty(M) $ is a finite-dimensional Lie group consisting of $C^\infty$ diffeomorphisms; 
    \item if $A$ is weakly mixing with respect to $m$, then $\Z^{\mathrm{Lip}}(A) \subset\mathrm{Aff}(\Gamma\backslash G) $ consists of affine automorphisms.
\end{itemize}

In this note, we establish an analogue of the result above for $A$-invariant complex structures. 
\subsection{Main results}

\begin{mainthm}\label{theorem main}
Let \(f\colon X\to X\) be a holomorphic diffeomorphism of a compact complex manifold $X$ with complex structure $J_X$. Let \(G\) be a simply connected real Lie group, \(\Gamma<G\) be a cocompact lattice, and $m$ be the Haar measure on $\Gamma\backslash G$. Suppose that there exists a bi-Lipschitz homeomorphism \(h\colon X\to \Gamma\backslash G\) such that
\[
        h\circ f=A\circ h
\]
for some affine automorphism \(A\colon \Gamma\backslash G\to \Gamma\backslash G\).
\begin{itemize}
    \item If $A$ is ergodic with respect to $m$, then $h$ is $C^\infty$, and the push-forward \(J=h_*J_X\) belongs to a finite-dimensional linear subspace $\mathcal E_A\subset C^\infty( \Gamma\backslash G,\operatorname{End}(\mathfrak g))$ depending only on \(A\). In particular,
\[
h:(X,J_X)\to(\Gamma\backslash G,J)
\]
is biholomorphic, and \(A\) is holomorphic with respect to \(J\). 
    \item Moreover, if $A$ is weakly mixing with respect to $m$, then \(J=h_*J_X\) is left-invariant. 
\end{itemize}
\end{mainthm}

In particular, by taking \(h=\operatorname{id}_{\Gamma\backslash G}\), we obtain the following corollary. \begin{coralph}\label{coro 1}
Let \(G\) be a simply connected real Lie group, \(\Gamma<G\) be a cocompact lattice, and $m$ be the Haar measure on $\Gamma\backslash G$. For any affine automorphism \(A\colon \Gamma\backslash G\to \Gamma\backslash G\), let $\mathcal{J}_{\Gamma\backslash G}(A)$ denote the set of all $A$-invariant complex structures on $\Gamma\backslash G$.
\begin{itemize}
    \item If $A$ is ergodic with respect to $m$, then $\mathcal{J}_{\Gamma\backslash G}(A)$ is contained in a finite-dimensional subspace $\mathcal E_A\subset C^\infty(\Gamma\backslash G,\operatorname{End}(\mathfrak g))$ depending only on \(A\);
    \item If $A$ is weakly mixing with respect to $m$, then $\mathcal{J}_{\Gamma\backslash G}(A)$ consists of left-invariant complex structures. 
\end{itemize}
\end{coralph}

\begin{remark}
Corollary \ref{coro 1} is not an existence statement. We do not assume that \(\Gamma\backslash G\) admits a complex structure, nor that it admits an \(A\)-invariant one. In particular, \(\mathcal{J}_{\Gamma\backslash G}(A)\)
may be empty, in which case the conclusions are vacuous.
\end{remark}
Moreover, the weak mixing assumption in Theorem \ref{theorem main} and Corollary \ref{coro 1} is essential. This mirrors the centralizer rigidity theorem of \cite{dwwx25}: when $A$ is ergodic but not weakly mixing with respect to $m$, Walters constructed examples \cite{wal68} of elements $f\in \Z^{\infty}(A)$ which are not affine (see also \cite[Section~3]{dwwx25}). This phenomenon also has an analogue in our setting, as shown in Section \ref{sec walters}. 
\begin{mainexample}\label{theorem example}
In the setting of Corollary~\ref{coro 1}, there exists an example for which $A$ is ergodic but not weakly mixing with respect to $m$, and $\mathcal{J}_{\Gamma\backslash G}(A)$ contains a complex structure that is not left-invariant.
\end{mainexample}
\subsection{An application to holomorphic Anosov diffeomorphisms}
Theorem~\ref{theorem main} can be applied to study holomorphic Anosov diffeomorphisms. Recall that a diffeomorphism \(f\colon M\to M\) of a compact Riemannian manifold \(M\) is called \textit{Anosov} if the tangent bundle admits a continuous \(Df\)-invariant splitting
\[
TM=E^s\oplus E^u
\]
and there exists a positive integer \(k\) such that, for every \(x\in M\),
\[
\|Df^k|_{E^s(x)}\|<1
\quad\text{and}\quad
\|Df^{-k}|_{E^u(x)}\|<1.
\]
All known examples of Anosov diffeomorphisms are topologically conjugate to
affine automorphisms of infra-nilmanifolds, and a long-standing conjecture
going back to Anosov and Smale asserts that there are no other examples
\cite{sma67}.

In general, a topological conjugacy need not be \(C^1\), since the Lyapunov exponents at corresponding periodic points may fail to coincide. Moreover, even when a \(C^1\) conjugacy exists, it need not be \(C^\infty\) \cite[Theorem~6.3]{lla92}. Nevertheless, for Anosov diffeomorphisms on tori that are \(C^1\)-conjugate to their linear models, the conjugacy is indeed \(C^\infty\) on $\TT^2$ \cite{mm87,lla92} and $\TT^3$ \cite{ks09,gog17,dg24}. More recently, Kalinin--Sadovskaya--Wang \cite{ksw23,ksw24} proved that if a \(C^\infty\) Anosov diffeomorphism is \(C^1\)-conjugate to a very weakly irreducible hyperbolic toral automorphism, then the conjugacy is \(C^\infty\). For general nilmanifolds, although irreducibility conditions can be formulated to study Lyapunov spectrum rigidity \cite{dew21}, no analogous bootstrap result seems to be known to the author's knowledge.

As a corollary of Theorem~\ref{theorem main}, we get a holomorphic bootstrap result. 
\begin{coralph}\label{coro single}
Let \(f\colon X\to X\) be a holomorphic Anosov diffeomorphism of a compact complex manifold \(X\). Suppose that \(f\) is conjugate to an affine Anosov automorphism \(A\) of a real infra-nilmanifold \(N\) via a bi-Lipschitz homeomorphism \(h:X\to N\). Then \(N\) admits a complex infra-nilmanifold structure \(J_N\) such that
\[
h\colon X\to (N,J_N)
\]
is biholomorphic and \(A\colon (N,J_N)\to (N,J_N)\) is a holomorphic affine automorphism.
\end{coralph}
\begin{remark}
The holomorphic setting leads to a somewhat different phenomenon: no irreducibility assumption is needed for the bootstrap. Moreover, the corollary applies to affine models on general infra-nilmanifolds, not only to toral models. 
\end{remark}

\subsection{An application to Anosov actions}
Theorem~\ref{theorem main} can also be applied to study holomorphic higher-rank
Anosov actions. Recall that an action
\[
\alpha\colon \mathbb Z^k\to \Diff^\infty(M),\quad k\geq 2,
\]
is called an \emph{Anosov action} if \(\alpha(\mathbf n)\) is an Anosov diffeomorphism for some \(\mathbf n\in \mathbb Z^k\). When \(M\) is homeomorphic to an infra-nilmanifold, the results of Franks \cite{fra70} and Manning \cite{man74}, together with the commutativity of the action and a result by Walters \cite{wal70}, imply that \(\alpha\) is topologically conjugate to an affine action
\[
\rho\colon \mathbb Z^k\to \mathrm{Aff}(N),
\]
where \(N\) is an infra-nilmanifold homeomorphic to \(M\). The action \(\rho\) is called the \emph{linearization} of \(\alpha\).

Hertz--Wang \cite{hw14} proved a global rigidity theorem for such actions: if \(M\) is homeomorphic to an infra-nilmanifold and the linearization \(\rho\) has no rank-one factor\footnote{This means that there
exists a subgroup \(\Sigma\subset \mathbb Z^k\) with
\(\Sigma\cong \mathbb Z^2\) such that \(\rho({\mathbf n})\) is ergodic for every
\(\mathbf n\in\Sigma\setminus\{0\}\). Other equivalent definitions can be found
in \cite[Section 2.5]{hw14}.}, then \(\alpha\) is \(C^\infty\)-conjugate to \(\rho\). In particular, \(M\) is diffeomorphic to an infra-nilmanifold.

Applying Theorem~\ref{theorem main}, we obtain the following holomorphic analogue of the Hertz--Wang theorem. 

\begin{coralph}\label{coro higher rank}
Let \(\alpha:\mathbb Z^k\curvearrowright X\) be a holomorphic Anosov action on a compact complex manifold \(X\). Suppose that \(X\) is homeomorphic to a real infra-nilmanifold \(N\), and that the linearization
\[
\rho\colon \mathbb Z^k\to \operatorname{Aff}(N)
\]
of \(\alpha\) has no rank-one factor. Then \(N\) admits a complex infra-nilmanifold structure \(J_N\). Moreover, the Franks--Manning conjugacy
\[
h\colon X\to (N,J_N)
\]
satisfying
\[
h\circ \alpha(\mathbf n)=\rho(\mathbf n)\circ h
\quad \forall\mathbf n\in\mathbb Z^k
\]
is biholomorphic. In particular, \(\rho\) is an action by holomorphic affine automorphisms of \((N,J_N)\).
\end{coralph}
\begin{remark}
Corollary \ref{coro higher rank} does not follow directly from Hertz--Wang \cite{hw14}, since, a priori, the underlying real infra-nilmanifold \(N\) may carry non-standard complex structures. Nevertheless, we use Hertz--Wang~\cite{hw14} to obtain the \(C^\infty\) conjugacy needed to apply Theorem~\ref{theorem main}.

For related works on holomorphic abelian actions on Kähler manifolds and projective varieties, 
see \cite{ds04,zha16}.
\end{remark}

\subsection{Further questions}

The relationship between Problem~\ref{problem 1} and Problem~\ref{problem 2} suggests that conjectures and theorems concerning centralizers and abelian actions should have holomorphic analogues. We give two examples here.

A guiding principle in the centralizer problem is that a large centralizer should force rigidity. One of the prominent conjectures reflecting this philosophy is the Katok--Spatzier conjecture.
\begin{conjecture}[Katok--Spatzier]
Let \(f\in \Diff^\infty(M)\) be an Anosov diffeomorphism. Suppose that the
\(C^\infty\)-centralizer \(\mathcal Z^\infty(f)\) is higher-rank. Then \(M\) is diffeomorphic to an infra-nilmanifold, and \(f\) is \(C^\infty\)-conjugate to an infra-nilmanifold automorphism.
\end{conjecture}

In view of the analogy between centralizers and invariant complex structures,
the natural holomorphic counterpart of the Katok--Spatzier conjecture is
Ghys' conjecture on holomorphic Anosov diffeomorphisms.

\begin{conjecture}[Ghys]
Let \(f\in \Diff^\infty(M)\) be an Anosov diffeomorphism. Suppose that \(\mathcal{J}_M(f)\neq\emptyset\), i.e. \(f\) preserves at least one complex structure. Then for every \(J\in \mathcal J_M(f)\), the complex manifold \((M,J)\) is biholomorphic to a complex infra-nilmanifold, and \(f\) is
biholomorphically conjugate to an infra-nilmanifold automorphism.
\end{conjecture}
On the other hand, trivial centralizers could be typical. A major problem in the study of centralizers is Smale's conjecture, which predicts that \(C^r\)-generically the centralizer is trivial, i.e.,  \(\mathcal Z^r(f)=\langle f\rangle\). This conjecture was proved in the \(C^1\)-topology by Bonatti--Crovisier--Wilkinson~\cite{bcw}. Motivated by the analogy above, one may ask for the following analogue of Smale's conjecture for invariant complex structures.
\begin{question}[An analogue of Smale's conjecture for invariant complex structures]
Let \(M\) be a closed smooth manifold admitting at least one complex structure. Does there exist a $C^1$-open and $C^\infty$-dense subset $ \mathcal U \subset \Diff^\infty(M)$ such that
\[
    \mathcal J_M(f)=\varnothing
\]
for every \(f\in \mathcal U\)? Equivalently, does every \(f\in \mathcal U\) fail to preserve any complex
structure on \(M\)?
\end{question}

\begin{remark}
As shown by Bonatti--Crovisier--Vago--Wilkinson \cite{bcvw}, the property
\(\mathcal Z^r(f)=\langle f\rangle\) is not \(C^1\)-open. Thus in the centralizer problem, one can at best expect such a property to hold generically, rather than on a \(C^1\)-open set.

By contrast, the question above has a straightforward answer in the \(C^1\)-open and \(C^1\)-dense setting. Indeed, by the \(C^1\) closing lemma \cite{pug67} and Franks' lemma \cite{fra71}, one can perturb a diffeomorphism so as to create a periodic point whose return map has a real simple eigenvalue. This is incompatible with the existence of an \(f\)-invariant complex structure.
\end{remark}

\subsection*{Acknowledgments.} The author is grateful to Danijela Damjanović for her encouragement and many helpful discussions, especially for explaining the work \cite{dwwx25} to the author. The author also thanks Ralf Spatzier, Amie Wilkinson, Chengyang Wu, Junyi Xie, and Disheng Xu for helpful discussions and valuable suggestions.

\section{Preliminaries }

\subsection{Almost complex structures and integrability}\label{sec complex structure}

Let \(M\) be a smooth real manifold. A smooth (continuous,
measurable) \textit{almost complex structure} on \(M\) is a smooth
(continuous, measurable) bundle endomorphism
\(J\colon TM\to TM\) such that
\[
J^2=-\operatorname{Id}.
\]
A manifold \(M\) endowed with a smooth almost complex structure \(J\) is called
an \textit{almost complex manifold}, and is denoted by \((M,J)\).

We say that a smooth almost complex structure \(J\) is \textit{integrable}
if it is induced by a complex manifold structure on \(M\), i.e., if there exists an atlas \(\{\phi_\alpha\colon U_\alpha\to \mathbb C^d\}\) of \(M\) whose transition maps are holomorphic and such that, for every \(\alpha\),
\[
D\phi_\alpha\circ J=i\circ D\phi_\alpha ,
\]
where \(i\) denotes the standard complex structure on \(\mathbb C^d\), i.e. multiplication by \(\sqrt{-1}\). In this case, $J$ is called the \textit{complex structure} of the complex manifold $M$. 

For later use, we record the following standard consequence of elliptic
regularity.

\begin{lemma}
\label{lemma weak CR regularity}
Let \((M,J)\) be a smooth \(2d\)-dimensional almost complex manifold. Let \(U\subset M\) be an open set, and let \(\phi\colon U\to \mathbb C^d\) be a Lipschitz map. Since \(\phi\) is differentiable almost everywhere by Rademacher's theorem, suppose that
\[
D\phi_y\circ J(y)=i\circ D\phi_y
\quad \text{for a.e. } y\in U,
\]
where \(i\) denotes the standard complex structure on \(\mathbb C^d\), i.e. multiplication by \(\sqrt{-1}\). Then \(\phi\) is smooth, and the above identity holds for every \(y\in U\).
\end{lemma}
\begin{proof}
Since \(\phi\) is Lipschitz, we have \(\phi\in W^{1,\infty}_{\mathrm{loc}}(U,\mathbb C^d)\). Define
\[
\bar\partial_J\phi
=
\frac12\bigl(D\phi+i\,D\phi\circ J\bigr).
\]
Then the assumption is precisely
\[
\bar\partial_J\phi=0
\]
in the weak sense. In local coordinates, the operator \(\bar\partial_J\) is a first-order elliptic linear differential operator with smooth coefficients. It follows from the standard regularity theorem for elliptic operators (see, for example,\cite[Chapter~IV, Section~4]{wells08}) that \(\phi\) is smooth. The identity then holds everywhere by continuity.
\end{proof}

\subsection{Left-invariant complex structure}
\label{section complex nilmanifold}
Let \(G\) be a connected real Lie group and let \(\mathfrak g\) be its Lie algebra. Any linear map \(J\in \operatorname{End}(\mathfrak g)\) with \(J^2=-\operatorname{id}\) defines a left-invariant almost complex structure on
\(G\) by left translations. By the Newlander--Nirenberg theorem \cite{nn57},
this almost complex structure is integrable if and only if its Nijenhuis tensor
\[
N_J(X,Y):=[JX,JY]-[X,Y]-J[JX,Y]-J[X,JY]
\]
vanishes for all \(X,Y\in \mathfrak g\). In this case, we say that \(J\) is a
left-invariant complex structure on \(G\). Moreover, if \(\Gamma\) is a
cocompact lattice in \(G\), then \(J\) descends to a left-invariant complex structure \(J_M\)
on \(M=\Gamma\backslash G\). 

In particular, if $M$ is a nilmanifold endowed with a left-invariant complex structure \(J_M\), then \((M,J_M)\) is called a \textit{complex nilmanifold}. If a finite group acts freely on \((M,J_M)\) by holomorphic affine automorphisms, then the quotient is called a \textit{complex infra-nilmanifold}. 


\section{Proofs}
The key idea in the proof of Theorem \ref{theorem main} is to view the pushed-forward complex structure \(h_*J_X\) as a bounded measurable map with values in \(\operatorname{End}_{\mathbb R}(\mathfrak g)\). The invariance of the complex structure then becomes a linear equivariance relation over the affine automorphism. We first record some elementary rigidity results for such equivariant maps.

\subsection{Equivariant tensor fields over affine automorphisms}

The argument in this subsection is essentially linear algebraic, and is inspired by the proof of \cite[Theorem B]{dwwx25}, where the authors prove the corresponding statement for derivative cocycles. Here we record a slightly more general version for tensor-valued equivariant maps. Although our application below concerns complex structures, the lemmas may also be useful for other tensor fields whose transformation laws are induced by \(DA\).

Let \(G\) be a real connected Lie group, and let \(m\) be the Haar measure on \(\Gamma\backslash G\), where \(\Gamma<G\) is a cocompact lattice. Let \(A\colon \Gamma\backslash G\to \Gamma\backslash G\) be an affine automorphism, and let \(DA\colon \mathfrak g\to \mathfrak g\) denote the Lie algebra automorphism induced by its linear part. Consider a bounded measurable map
\[
B\colon \Gamma\backslash G\to \operatorname{End}_{\mathbb R}(\mathfrak g)
\]
satisfying
\begin{equation}\label{eq:invariant}
    B(Ax)=DA\circ B(x)\circ (DA)^{-1}  \quad\text{for } m \text{-a.e. }x \in \Gamma\backslash G.
\end{equation}

The next two lemmas describe the rigidity of such maps under ergodicity and
weak mixing assumptions on \(A\).
\begin{lemma}\label{lemma cocycle ergodic}
If \(A\) is \(m\)-ergodic, then there exists a finite-dimensional subspace
\[
\mathcal E_A\subset C^\infty(\Gamma\backslash G,\operatorname{End}_{\mathbb R}(\mathfrak g)),
\]
depending only on \(A\), such that every \(B\) satisfying \eqref{eq:invariant} agrees almost everywhere with an element of \(\mathcal E_A\).
\end{lemma}

\begin{proof}
We first complexify \eqref{eq:invariant}. Denote $V=\operatorname{End}_{\mathbb C}(\mathfrak g_{\mathbb C})$, and define a linear transformation
\[
T\colon V\to V,
\qquad
T(C)=DA_{\mathbb C}\circ C\circ (DA_{\mathbb C})^{-1}.
\]
Then equation \eqref{eq:invariant} becomes
\[
B_{\mathbb C}\circ A=T( B_{\mathbb C}).
\]
Decompose \(V\) into the direct sum of the generalized eigenspaces of \(T\):
\[
V=\bigoplus_{\rho\in \operatorname{Spec}(T)}V_\rho.
\]
Let \(P_\rho:V\to V_\rho\) be the spectral projection and set \(B_{\mathbb C,\rho}=P_\rho\circ B_{\mathbb C}\). Since \(P_\rho\) commutes with \(T\), applying \(P_\rho\) to \(B_{\mathbb C}\circ A=T(B_{\mathbb C})\) gives
\[
B_{\mathbb C,\rho}\circ A=T( B_{\mathbb C,\rho}),
\]
and hence
\[
B_{\mathbb C,\rho}(A^n x)=T^nB_{\mathbb C,\rho}(x)
\]
for every \(n\geq 0\) and for almost every \(x\). More precisely, after replacing the original full-measure set by the intersection of its inverse images under all iterates of \(A\), we may assume that this relation holds along the entire forward orbit of \(x\).

If \(|\rho|>1\), then \(T^n v\) is unbounded for every nonzero \(v\in V_\rho\). Since \(B_{\mathbb C,\rho}\) is essentially bounded, it follows that \(B_{\mathbb C,\rho}=0\) almost everywhere. For \(|\rho|<1\), applying the same argument to \(A^{-1}\) gives the same conclusion. Thus only eigenvalues \(\rho\) with \(|\rho|=1\) may occur.

Now assume \(|\rho|=1\). Write
\[
T|_{V_\rho}=\rho(I+N_\rho),
\]
where \(N_\rho\) is nilpotent. For almost every \(x\), the sequence
\[
T^nB_{\mathbb C,\rho}(x)
=
\rho^n(I+N_\rho)^nB_{\mathbb C,\rho}(x)
\]
is bounded. Hence \(N_\rho B_{\mathbb C,\rho}(x)=0\) almost everywhere; otherwise the right-hand side would have polynomial growth. Therefore
\[
B_{\mathbb C,\rho}(x)\in \ker(T-\rho I)
\]
for almost every \(x\). 

Choose a basis \(C_{\rho,1},\ldots,C_{\rho,d_\rho}\) of \(\ker(T-\rho I)\). Then
\[
B_{\mathbb C,\rho}(x)
=
\sum_{j=1}^{d_\rho}\varphi_{\rho,j}(x)C_{\rho,j}
\]
for some bounded measurable functions \(\varphi_{\rho,j}\) with
\[
\varphi_{\rho,j}\circ A=\rho\varphi_{\rho,j}.
\]
Thus \(\varphi_{\rho,j}\in \mathcal K_\rho\), where
\[
\mathcal K_\rho
=
\{\varphi\in L^\infty(\Gamma\backslash G,m):\varphi\circ A=\rho\varphi\}.
\]
By \cite[Theorem~20]{dwwx25}, for an ergodic affine automorphism \(A\), every nonzero \(L^2\)-eigenspace
\[
E_\rho=\{\varphi\in L^2(\Gamma\backslash G,m):\varphi\circ A=\rho\varphi\}
\]
is one-dimensional and is spanned by a smooth eigenfunction. Since \(\mathcal K_\rho\subset E_\rho\), every nonzero \(\mathcal K_\rho\) is also spanned by a smooth eigenfunction. Consequently \(B_{\mathbb C}\) agrees almost
everywhere with an element of the finite-dimensional space
\[
\mathcal E_{\mathbb C,A}
=
\bigoplus_{\substack{\rho\in\operatorname{Spec}(T)\\ |\rho|=1}}
E_\rho\otimes \ker(T-\rho I)
\subset
C^\infty(\Gamma\backslash G,\operatorname{End}_{\mathbb C}(\mathfrak g_{\mathbb C})),
\]
which depends only on \(A\).

Let \(F\in\mathcal E_{\mathbb C,A}\) be such that \(F=B_{\mathbb C}\) almost everywhere. Since \(B_{\mathbb C}\) is the complexification of a real tensor, \(F\) preserves the real form \(\mathfrak g\subset\mathfrak g_{\mathbb C}\) almost everywhere. By smoothness, this holds everywhere. Hence $F=F_{0,\mathbb C}$ is the complexification for a unique real tensor
\[
F_0\in C^\infty(\Gamma\backslash G,\operatorname{End}_{\mathbb R}(\mathfrak g)),
\]
and \(B=F_0\) almost everywhere. Then
\[
\mathcal E_A
=
\{F_0: F_{0,\mathbb C}\in \mathcal E_{\mathbb C,A}\}
\]
is the finite-dimensional space we require. 
\end{proof}
\begin{lemma}\label{lemma cocycle mixing}
If \(A\) is weakly mixing with respect to \(m\), then every \(B\) satisfying \eqref{eq:invariant} is equal to a constant $B_0 \in \operatorname{End}_{\mathbb R}(\mathfrak g)$ almost everywhere. Moreover, $B_0$ satisfies
\[DA\circ B_0=B_0\circ DA.\]
\end{lemma}
\begin{proof}
Since weak mixing implies ergodicity, the proof of Lemma~\ref{lemma cocycle ergodic} also applies in this setting. In particular, for every generalized eigenspace \(V_\rho\) of \(T\), the component \(B_{\mathbb C,\rho}=P_\rho\circ B_{\mathbb C}\) satisfies the following properties:
\begin{itemize}
    \item \(B_{\mathbb C,\rho}(x)=0\) for almost every $x $ whenever \(|\rho|\neq 1\). 
    \item $B_{\mathbb C,\rho}(Ax)=\rho B_{\mathbb C,\rho}(x)$ for almost every $x$ when \(|\rho|=1\). 
\end{itemize}
We then show that \(B_{\mathbb C,\rho}=0\) for every \(\rho\neq 1\). For any given $\rho \neq 1$, let \(\ell\in V_\rho^\ast\) be a linear functional on \(V_\rho\), and define
\[
\varphi(x):=\ell(B_{\mathbb C,\rho}(x)).
\]
Then \(\varphi:\Gamma\backslash G \to \CC\) is a bounded measurable function with
\begin{equation}\label{eq varphi}
\varphi(Ax)
=
\ell(B_{\mathbb C,\rho}(Ax))
=
\ell(\rho B_{\mathbb C,\rho}(x))
=
\rho\varphi(x).
\end{equation}
Since \(A\) is weakly mixing with respect to \(m\), we have
\begin{equation}\label{eq weakly mixing}
    \frac1N\sum_{n=0}^{N-1}
\left|
\int_{\Gamma\backslash G} \varphi(A^n x)\overline{\varphi(x)}\,dm
-
\left(\int_{\Gamma\backslash G} \varphi\,dm\right)
\overline{\left(\int_{\Gamma\backslash G} \varphi\,dm\right)}
\right|
\longrightarrow 0 .
\end{equation}
Since \(m\) is \(A\)-invariant, integrating equation \eqref{eq varphi} gives
\[
\int_{\Gamma\backslash G} \varphi\,dm
=
\rho \int_{\Gamma\backslash G} \varphi\,dm.
\]
As \(\rho\neq 1\), it follows that
\begin{equation}\label{eq term2}
    \int_{\Gamma\backslash G} \varphi\,dm=0.
\end{equation}
On the other hand, equation \eqref{eq varphi} also implies, for every \(n\geq 0\),
\begin{equation}\label{eq term1}
\int_{\Gamma\backslash G} \varphi(A^n x)\overline{\varphi(x)}\,dm
=
\rho^n\int_{\Gamma\backslash G} |\varphi(x)|^2\,dm .
\end{equation}
Since $|\rho|=1$, substituting \eqref{eq term2} and \eqref{eq term1} into \eqref{eq weakly mixing} gives
\[
\int_{\Gamma\backslash G} |\varphi(x)|^2\,dm=0.
\]
Thus \(\varphi(x)=0\) for almost every \(x\). Since this holds for every \(\ell\in V_\rho^\ast\) and $\rho \neq 1$, we conclude that
\[
B_{\mathbb C,\rho}(x)=0
\]
for every $\rho \neq 1$ and almost every $x$. Therefore, 
\[
B_{\mathbb C}(Ax)=B_{\mathbb C}(x)
\]
is $A$-invariant. Since \(A\) is weakly mixing, it is ergodic. Hence \(B_{\mathbb C}\) is constant almost everywhere, and \(B\) is also almost everywhere equal to a real constant \(B_0\). Substituting this constant into \eqref{eq:invariant} gives \(DA\circ B_0=B_0\circ DA\).
\end{proof}

\subsection{A lemma on push-forward complex structure}
We prove a bootstrap lemma concerning the push-forward complex structure. Let \((X,J_X)\) be a complex manifold and let \(N\) be a smooth real manifold. If \(h\colon X\to N\) is a bi-Lipschitz homeomorphism, then both \(h\) and \(h^{-1}\) are differentiable almost everywhere. At points where both differentials exist, the chain rule gives \(D(h^{-1})_{h(x)}=(Dh_x)^{-1}\).  Thus 
\[
(h_*J_X)(y)
=
Dh_x\circ J_X(x)\circ (Dh_x)^{-1},
\qquad y=h(x),
\]
defines an \(L^\infty\) almost complex structure on \(N\). The following lemma
says that if the measurable push-forward \(h_\ast J_X\) agrees almost everywhere with a smooth almost complex structure, then that smooth almost complex structure is integrable.
\begin{lemma}\label{lemma complex structure}
If there exists a smooth almost complex structure \(J_N\) on \(N\) such that
\[
J_N=h_*J_X
\quad \text{for a.e. }y \in N
\]
then \(J_N\) is integrable and
\[
h\colon (X,J_X)\to (N,J_N)
\]
is biholomorphic. In particular, the homeomorphism $h$ is $C^\infty$ and $J_N=h_*J_X$ for every $y \in N$. 
\end{lemma}
\begin{proof}
Let \(\phi_\alpha:U_\alpha\to\mathbb C^d\) be holomorphic coordinates for \((X,J_X)\). Set
\[
V_\alpha=h(U_\alpha),
\qquad
\psi_\alpha:=\phi_\alpha\circ h^{-1}:V_\alpha\to\mathbb C^d .
\]
The maps \(\psi_\alpha\) are bi-Lipschitz and the transition maps
\[
\psi_\alpha\circ\psi_\beta^{-1}
=
\phi_\alpha\circ h^{-1}\circ h\circ\phi_\beta^{-1}
=
\phi_\alpha\circ\phi_\beta^{-1},
\]
are holomorphic. 

Since  \(h_*J_X\) agrees almost everywhere with  \(J_N\), for almost every \(y=h(x)\), we have
\[
D\psi_\alpha|_y\circ J_N(y)
=
D\psi_\alpha|_y\circ (h_*J_X)(y).
\]
By the definition of \(h_*J_X\), this gives
\[
D\psi_\alpha|_y\circ J_N(y)
=
D\phi_\alpha|_x\circ (Dh_x)^{-1}
\circ Dh_x\circ J_X(x)\circ (Dh_x)^{-1}.
\]
Thus
\[
D\psi_\alpha|_y\circ J_N(y)
=
D\phi_\alpha|_x\circ J_X(x)\circ (Dh_x)^{-1}.
\]
Since \(\phi_\alpha\) is \(J_X\)-holomorphic,
\[
D\phi_\alpha|_x\circ J_X(x)
=
i\circ D\phi_\alpha|_x .
\]
Therefore, 
\[
D\psi_\alpha|_y\circ J_N(y)
=
i\circ D\phi_\alpha|_x\circ(Dh_x)^{-1}
=
i\circ D\psi_\alpha|_y
\quad \text{for a.e. }y\in V_\alpha .
\]
By Lemma~\ref{lemma weak CR regularity}, each \(\psi_\alpha\) is smooth and the above identity holds everywhere.  Since \(\psi_\alpha\) is bi-Lipschitz, there is \(c>0\) such that \(|\psi_\alpha(p)-\psi_\alpha(q)|\ge c\,d(p,q)\) locally. Because \(\psi_\alpha\) is now smooth, this implies that \(D\psi_\alpha\) is injective at every point. By dimension reasons \(D\psi_\alpha\) is invertible, and the inverse function theorem shows that the \(\psi_\alpha\) are smooth coordinate charts. Hence the charts \(\psi_\alpha\) are smooth \(J_N\)-holomorphic coordinates, and \(J_N\) is integrable.

Finally, by construction of the pushed-forward charts \(\psi_\alpha=\phi_\alpha\circ h^{-1}\), the coordinate representation of \(h\) in the charts \(\phi_\alpha\) on \(X\) and \(\psi_\alpha\) on \(N\) is
\[
\psi_\alpha\circ h\circ \phi_\alpha^{-1}
=
\phi_\alpha\circ h^{-1}\circ h\circ \phi_\alpha^{-1}
=
\operatorname{id}.
\]
Since the charts \(\psi_\alpha\) have just been shown to be \(J_N\)-holomorphic coordinates, it follows that \(h\) is biholomorphic. In particular, the homeomorphism $h$ is $C^\infty$ and $J_N=h_*J_X$ for every $y \in N$. 
\end{proof}
\subsection{Proof of Theorem \ref{theorem main}}
Since \(h\) is bi-Lipschitz, both \(h\) and \(h^{-1}\) are differentiable almost everywhere. At points where both differentials exist, the chain rule gives \(D(h^{-1})_{h(x)}=(Dh_x)^{-1}\). Thus, for almost every \(y\in \Gamma\backslash G\) and \(x=h^{-1}(y)\),
\[
(h_*J_X)(y)
=
Dh_{x}\circ J_X(x)\circ
(Dh_{x})^{-1}
\]
is well-defined for almost every \(y\in \Gamma\backslash G\), and gives an \(L^\infty\) almost complex structure on \(\Gamma\backslash G\). 

To apply Lemma \ref{lemma cocycle ergodic}, we first check that $h_\ast J_X$ satisfies \eqref{eq:invariant}.
\begin{lemma}
For almost every $y \in \Gamma\backslash G$, we have
\[
h_*J_X (Ay)
=
DA\circ h_*J_X (y)\circ DA^{-1}.
\]
\end{lemma}
\begin{proof}
Choose \(x\) in a full-measure set such that \(h\) is differentiable at both \(x\) and \(fx\), and such that \(h^{-1}\) is differentiable at both \(h(x)\) and \(h(fx)\). Differentiating the conjugacy relation \(h\circ f=A\circ h\) at such a point, we obtain
\[Dh_{fx}\circ Df_x=DA_y\circ Dh_x .\]
By the holomorphicity of \(f\), we have
\[Df_x\circ J_X(x)=J_X(fx)\circ Df_x.\]
Therefore, 
\[\begin{aligned}
DA_y\circ (h_*J_X)(y)
&=DA_y\circ Dh_x\circ J_X(x)\circ (Dh_x)^{-1} \\
&=Dh_{fx}\circ Df_x\circ J_X(x)\circ (Dh_x)^{-1} \\
&=Dh_{fx}\circ J_X(fx)\circ Df_x\circ (Dh_x)^{-1}\\
&=Dh_{fx}\circ J_X(fx)\circ (Dh_{fx})^{-1}\circ DA_y=(h_*J_X)(Ay)\circ DA_y.
\end{aligned}\]
This proves the lemma. 
\end{proof}
Using the left-invariant trivialization, we identify the tangent bundle \(T(\Gamma\backslash G)\) with \(\Gamma\backslash G\times\mathfrak g\). Then \(h_*J_X\) is represented by a measurable map
\[
h_*J_X \colon \Gamma\backslash G\to \operatorname{End}(\mathfrak g).
\]
By Lemma~\ref{lemma cocycle ergodic}, there exists a finite-dimensional subspace $\mathcal E_A\subset C^\infty(\Gamma\backslash G,\operatorname{End}_\mathbb R(\mathfrak g))$ depending only on \(A\), and an element \(J_{\Gamma\backslash G}\in \mathcal E_A\), such that
\[
h_*J_X(y)=J_{\Gamma\backslash G}(y),
\qquad \text{for }m\text{-a.e. }y\in \Gamma\backslash G.
\]
Since \(J_{\Gamma\backslash G}\) agrees almost everywhere with \(h_*J_X\) and \((h_*J_X)^2=-\operatorname{Id}\), smoothness implies \(J_{\Gamma\backslash G}^2=-\operatorname{Id}\) everywhere. Hence \(J_{\Gamma\backslash G}\) is a smooth almost complex structure. By Lemma~\ref{lemma complex structure}, \(J_{\Gamma\backslash G}\) is integrable, and \(h_*J_X\) agrees everywhere with \(J_{\Gamma\backslash G}\). Moreover,
\[
h\colon (X,J_X)\to (\Gamma\backslash G,J_{\Gamma\backslash G})
\]
is biholomorphic. In particular, the affine automorphism
\[
A=h\circ f\circ h^{-1}: (\Gamma\backslash G,J_{\Gamma\backslash G}) \to (\Gamma\backslash G,J_{\Gamma\backslash G})
\]
is biholomorphic and \(f\) is biholomorphically conjugate to $A$. Moreover, if $A$ is weakly mixing, by Lemma \ref{lemma cocycle mixing}, the complex structure $J_{\Gamma\backslash G}$ is constant in the left-invariant trivialization, and hence left-invariant.

\subsection{Proof of Corollary~\ref{coro single}}
Pass to a finite nilmanifold cover \(\pi_N\colon \widetilde N\to N\). After replacing it by an \(A\)-invariant finite cover if necessary, the affine Anosov automorphism \(A\) lifts to an affine Anosov automorphism \(\widetilde A\) of \(\widetilde N\). Taking the corresponding finite cover \(\pi_X:\widetilde X\to X\), the conjugacy \(h\) lifts to a bi-Lipschitz conjugacy \(\widetilde h\), and \(f\) lifts to a holomorphic Anosov diffeomorphism \(\widetilde f\) such that
\[
\widetilde h\circ \widetilde f=\widetilde A\circ \widetilde h .
\]

The map \(\widetilde A\) is an affine Anosov automorphism of the nilmanifold \(\widetilde N\), and hence is mixing with respect to Haar measure by the standard mixing criterion for nilmanifold automorphisms. Applying Theorem~\ref{theorem main}, we obtain a left-invariant complex structure \(J_{\widetilde N}\) on \(\widetilde N\) such that
\[
\widetilde h\colon \widetilde X\to (\widetilde N,J_{\widetilde N})
\]
is biholomorphic.

It remains to descend \(J_{\widetilde N}\) to \(N\). For every deck transformation \(\delta_N\) of \(\pi_N\), there is a deck transformation \(\delta_X\) of \(\pi_X\) satisfying \(\widetilde h\circ\delta_X=\delta_N\circ\widetilde h\). Since the lifted complex structure on \(\widetilde X\) is invariant under \(\delta_X\), its push-forward \(J_{\widetilde N}\) is invariant under \(\delta_N\). Therefore the deck group acts on \((\widetilde N,J_{\widetilde N})\) by holomorphic affine automorphisms, and \(J_{\widetilde N}\) descends to a complex infra-nilmanifold structure \(J_N\) on \(N\). This finishes the proof.

\subsection{Proof of Corollary~\ref{coro higher rank}}
As in the proof of Corollary~\ref{coro single}, after passing to a finite nilmanifold cover and then descending the resulting complex structure, it suffices to consider the case where \(N\) is a nilmanifold.

By the global rigidity theorem of Hertz--Wang~\cite{hw14}, the action \(\alpha\) is \(C^\infty\)-conjugate to its linearization 
\[
\rho\colon \mathbb Z^k\to \operatorname{Aff}(N)
\]
by a \(C^\infty\) diffeomorphism \(h\colon X\to N\). Choose \(\mathbf n_0\in\mathbb Z^k\) such that \(\alpha(\mathbf n_0)\) is Anosov.
Then \(\rho(\mathbf n_0)\) is an affine Anosov automorphism, and hence is mixing with respect to the Haar measure. Applying Theorem~\ref{theorem main} to \(\alpha(\mathbf n_0)\) and \(\rho(\mathbf n_0)\), we obtain a complex nilmanifold structure \(J_N\) on \(N\) such that
\[
h\colon X\to (N,J_N)
\]
is biholomorphic.

It remains to check that the whole affine action is holomorphic with respect
to \(J_N\). For every \(\mathbf n\in\mathbb Z^k\), we have
\[
\rho(\mathbf n)=h\circ \alpha(\mathbf n)\circ h^{-1}.
\]
Since \(\alpha(\mathbf n)\) is holomorphic and \(h\) is biholomorphic, it follows that \(\rho(\mathbf n)\) is holomorphic for every \(\mathbf n\in\mathbb Z^k\).

\section{An ergodic but not weakly mixing example}\label{sec walters}
We include the following example to show that the weak mixing assumption cannot be replaced by ergodicity in Theorem~\ref{theorem main} and Corollary~\ref{coro 1}. The construction is inspired by an example of Walters~\cite{wal68}; see also \cite[Section~3]{dwwx25}.
\begin{example}\label{example:ergodic-not-enough}
Let $ M=\TT^{10}=\TT^4\times \TT^4\times \TT^2$ with coordinates \((x_1,x_2,t,s)\). Endow \(M\) with the constant complex structure \(J_0\) defined by
\[
        J_0(u_1,u_2,a,b)=(-u_2,u_1,-b,a),
        \qquad
        (u_1,u_2,a,b)\in
        \RR^4\oplus \RR^4\oplus \RR\oplus \RR .
\]

Let \(L\in \SL(4,\ZZ)\) be an ergodic toral automorphism with a pair of complex conjugate eigenvalues $ \lambda=e^{i\theta}$ and $\bar{\lambda}=e^{-i\theta}$ that are not roots of unity. Let \(V=V_{\lambda,\bar\lambda}\subset \RR^4\) be the real two-dimensional invariant subspace corresponding to the pair \(\lambda,\bar\lambda\). After choosing linear coordinates on \(V\), we may identify \(L|_V\) with the rotation \(R_\theta\), where \(R_\eta\) denotes rotation by angle \(\eta\). Write $ \alpha=\frac{\theta}{2\pi}$ and choose \(\beta\in \RR/\ZZ\) such that \(1,\alpha,\beta\) are linearly independent over \(\QQ\). Define
\[
A\colon M\to M ,\qquad(x_1,x_2,t,s)
        \mapsto
        \bigl(Lx_1,Lx_2,t+\alpha,s+\beta\bigr).
\]
Then \(A\) is an ergodic holomorphic affine automorphism of \((M,J_0)\), but it is not weakly mixing (see, for example, \cite{knbook}).

Choose a non-zero vector \(v\in V\), and define a smooth \(\RR^4\)-valued periodic function
\[
        \tau(t)=R_{2\pi t}v.
\]
Let \(\bar\tau\colon \TT^1\to \TT^4\) be its projection modulo \(\ZZ^4\). Define
\[
        g(x_1,x_2,t,s)
        =
        \bigl(x_1+\bar\tau(t),x_2,t,s\bigr).
\]
Since \(L|_{V}=R_\theta\), we have
\[
\tau(t+\alpha)=R_{2\pi(t+\alpha)}v
=R_\theta R_{2\pi t}v
=L\tau(t),
\]
and therefore \(g\circ A=A\circ g\).

Define the pull-back complex structure \(J:=g^*J_0\), given by
\[
        J_x
        =
        (Dg_x)^{-1}\circ (J_0)_{g(x)}\circ Dg_x .
\]
Since $g$ commutes with $A$ and \(A^*J_0=J_0\), it follows that
\[
        A^*J
        =
        A^*g^*J_0
        =
        g^*A^*J_0
        =
        J .
\]
Thus \(J\) is \(A\)-invariant, and \(A\) is holomorphic with respect to \(J\).

However, \(J\) is not left-invariant with respect to the standard Lie group structure on \(M\). Indeed,
\[
        Dg(\partial_t)=\partial_t+\tau'(t),
        \qquad
        Dg(\partial_s)=\partial_s,
\]
and therefore
\[
        J(\partial_s)
        =
        Dg^{-1}\bigl(J_0(Dg(\partial_s))\bigr)
        =
        -\partial_t+\tau'(t).
\]
Since \(\tau'(t)\) is not constant, the coefficients of \(J\) depend non-trivially on \(t\). Thus \(J\) is not left-invariant.
\end{example}\begin{remark}
The complex manifold \((\TT^{10},J)\) constructed above is still biholomorphic to a complex torus, but the biholomorphism is given by the non-affine map \(g\). This is analogous to Walters's example discussed in \cite[Section~3]{dwwx25}: the element in the centralizer is not affine with respect to the original homogeneous structure, but becomes affine after
changing the homogeneous structure. 

One may ask whether a complex structure preserved by an ergodic affine automorphism can become left-invariant after a suitable change of the compatible homogeneous structure. The weak mixing assumption in Theorem~\ref{theorem main} gives precisely the stronger conclusion: the complex structure is already left-invariant without changing the homogeneous structure.
\end{remark}

\bibliographystyle{plain}
\bibliography{ref}
\end{document}